\documentclass[journal,twoside,web]{ieeecolor}
 
\let\labelindent\relax
\usepackage{lcsys}
\usepackage{generic}
\usepackage[colorlinks=true,linkcolor=blue, filecolor=blue, urlcolor=blue,citecolor=blue]{hyperref}
\usepackage{amsmath,amssymb,amsthm,epsfig,enumitem,mathtools,mathrsfs,bm}
\usepackage{times}
\usepackage{mathptmx} 
\usepackage{textcomp}
\usepackage{graphicx}
\usepackage{caption}
\usepackage{subfig}
\usepackage{float}
\usepackage{cite} 

\usepackage[dvipsnames]{xcolor}
\usepackage[font=small]{caption}
\usepackage{tikz}
\usepackage{pgfplots}
\usepackage{pgfplotstable}
\usepgfplotslibrary{groupplots}
\pgfplotsset{compat=1.18}
\usetikzlibrary{arrows.meta,positioning,calc,decorations.pathreplacing}
\definecolor{pOne}{RGB}{0,90,181}      
\definecolor{pTwo}{RGB}{220,50,32}     
\definecolor{shared}{RGB}{80,80,80} 

\theoremstyle{plain}
\newtheorem{theorem}{Theorem}

\newtheorem{lemma}{Lemma}

\theoremstyle{definition}
\newtheorem{definition}{Definition}
\newtheorem{assumption}{Assumption}
\theoremstyle{remark}
\newtheorem{remark}{Remark}
 
\newcommand{\Z}{\mathrm{0}}
\newcommand{\I}{\mathrm{I}} 
\newcommand{\rank}{{\rm rank}}
\newcommand{\N}{\mathsf{N}}
\newcommand{\K}{\mathsf{K}}
\newcommand{\su}{\bm{u}}
\newcommand{\col}[1]{\mathtt{col}\{#1\}}  
\newcommand{\row}[1]{\mathtt{row}\{#1\}}  
\newcommand{\Bdg}{\mathtt{Bdg}}  

\DeclareMathOperator*{\argmin}{arg\,min}

\begin{document} 
 \title{Open-Loop Stackelberg LQ Difference Games with Coupled-Affine Inequality Constraints:\\
 	An Exact OCP--LCS/LCQP Reformulation}
	\author{Shrestha Ghosh and Puduru Viswanadha Reddy, \IEEEmembership{Member, IEEE}
		\thanks{S. Ghosh and P. V. Reddy are with the Department of Electrical
			Engineering, Indian Institute of Technology Madras, Chennai, India.
			(e-mail: ee25s002@smail.iitm.ac.in, vishwa@ee.iitm.ac.in).}}	
\maketitle
\thispagestyle{empty}
	\begin{abstract}In this letter, we study finite-horizon linear--quadratic Stackelberg
		difference games with coupled-affine state-control inequality constraints.
		Under the stated assumptions, we show that generalized open-loop Stackelberg
		equilibria admit an exact reformulation as an optimal control problem subject
		to a discrete-time linear complementarity system. Eliminating the dynamic
		variables yields a large-scale linear complementarity quadratic program, together with an explicit recovery map for the follower
		strategy. This reformulation enables the numerical computation of these
		equilibria. We illustrate the proposed approach using a constrained
		network-flow game.
	\end{abstract}
	

\section{Introduction} 
Hierarchical decision making arises in multi-agent systems in which one
decision maker, the leader, commits to a strategy and another, the follower,
responds after observing that commitment. Stackelberg games provide a
framework for such sequential interactions \cite{Stackelberg:52}, with
applications in demand response, smart-grid energy management, and
hierarchical electric-vehicle charging \cite{Mah13,Che16b,Yoo16}. In dynamic
games, the Stackelberg solution concept depends on the information structure.
Under open-loop information, the leader announces a complete time-indexed
strategy and the follower responds over a   horizon; under feedback
information, the players select state-dependent policies. Both information
structures have been studied extensively for unconstrained dynamic games
\cite{Che72, Basar:1999,Xu:2015}.

Practical dynamic games involve state and control constraints, including
capacity, safety, saturation, and budget limits. In Stackelberg dynamic games
with coupled inequality constraints, the follower's feasible set
depends on the leader's announcement, while the induced follower response may
affect the leader's feasibility. As a result, the leader's anticipation of
the follower's response yields a constrained bilevel optimal control problem.
The central technical challenge is to characterize the follower's constrained
response and the induced leader-feasibility conditions.

Existing  approaches in constrained  Stackelberg dynamic games either restrict
the coupling structure or adopt a different equilibrium concept. The open-loop
formulation in \cite{Mon:19} assumes an unconstrained follower and reduces
the equilibrium conditions to a  linear complementarity problem
(LCP). The   generalized open-loop Nash formulation in
\cite{Moh23} is simultaneous rather than sequential. The model in
\cite{Moh25} studies a restricted feedback
Stackelberg--Nash game with sequential and simultaneous interactions, rather
than the fully coupled sequential setting. The work in \cite{Li24} studies
approximate local feedback Stackelberg equilibria. To the best of our
knowledge, an exact characterization of finite-horizon open-loop
Stackelberg equilibria for linear--quadratic difference games (LQDGs) with
coupled-affine inequality constraints remains lacking.

To address this gap, we study the above class of dynamic games. In
particular, this letter makes three contributions. First, we formulate the
generalized open-loop Stackelberg equilibrium (GOLSE) problem through a
follower-feasible reaction set and a leader-admissible strategy  set that
incorporates constraints induced by the follower's response. Second, a
Riccati-based costate transformation converts the follower's KKT conditions
into a coupled discrete-time linear complementarity system (LCS). Embedding
this system in the leader's problem yields an optimal-control problem subject
to LCS constraints (OCP--LCS), for which we establish an exact correspondence
with the GOLSE formulation and an affine recovery map for the follower
strategy. Third, eliminating the state and costate trajectories yields an
initial-state-parametrized linear complementarity quadratic program (LCQP),
in contrast to the restricted or different Stackelberg structures considered
in \cite{Mon:19,Moh25}, where the equilibrium conditions reduce to an LCP.
The resulting LCQP is generally nonconvex due to its complementarity
constraints. We show that its global minimizers correspond to GOLSEs in the
primal decision variables.

The remainder of the letter is organized as follows. Section~\ref{sec:formulation}
formulates the constrained linear--quadratic difference game and defines the
generalized open-loop Stackelberg equilibrium. Section~\ref{sec:GOLSE} derives
the follower reduction and the OCP--LCS and LCQP reformulations.
Section~\ref{sec:numerical} presents numerical results, and
Section~\ref{sec:conclusion} concludes the letter.

\subsubsection*{Notation} 
Let \(\mathbb{R}^n\), \(\mathbb{R}^{n\times m}\), and \(\mathbb{S}^n\)
denote the sets of real \(n\)-vectors, real \(n\times m\) matrices, and
real symmetric \(n\times n\) matrices, respectively. For a matrix \(A\),
\(A^\top\) and \(\mathrm{rank}(A)\) denote its transpose and rank;
\(\I_n\) denotes the \(n\times n\) identity matrix, and \(\Z\) a zero
matrix of appropriate dimension. For \(A\in\mathbb{S}^n\), \(A\succ0\)
denotes positive definiteness. For \(x\in\mathbb{R}^n\) and
\(A\in\mathbb{S}^n\), \(\|x\|_A^2:=x^\top A x\), and \([x]_i\) denotes the
\(i\)th component of \(x\). The operators
\(\col{A_i}_{i=1}^N\), \(\row{A_i}_{i=1}^N\), and
\(\Bdg\{A_i\}_{i=1}^N\) denote the block column, block row, and block
diagonal matrix formed from \(\{A_i\}_{i=1}^N\), respectively. Vectors
\(x,y\in\mathbb{R}^n\) are complementary if \(x\geq0\), \(y\geq0\), and
\(x^\top y=0\), denoted by \(0\leq x\perp y\geq0\).
\section{Problem Formulation}
\label{sec:formulation}
In this section, we describe the class of LQDGs with coupled inequality
constraints and define the associated open-loop Stackelberg equilibrium.
Let $\N:=\{1,2\}$ denote the set of players, where player $1$ is the
leader and player $2$ is the follower. Let
$\K:=\{0,1,\ldots,K\}$ be the set of decision instants and define
$\K_l:=\K\setminus\{K\}$ and $\K_r:=\K\setminus\{0\}$. At each $k\in\K_l$,
player $i\in\N$ selects an action $u_k^i\in\mathbb{R}^{m_i}$ to influence
the evolution of the state vector $x_k\in\mathbb{R}^{n}$ according to the
following discrete--time dynamics
\begin{subequations}
	\begin{align}
		x_{k+1}&=A_kx_k+B_k^1u_k^1+B_k^2u_k^2,
		\quad x_0=\bar{x}_0,     \label{eq:state_dyn}
	\end{align}
	where $A_k\in\mathbb{R}^{n\times n}$, $B_k^i\in\mathbb{R}^{n\times m_i}$,
	and $\bar{x}_0\in\mathbb{R}^n$ is the initial state. For each $i\in\N$,
	$\su^i:=\col{u_k^i}_{k\in\K_l}\in\mathbb{R}^{Km_i}$ denotes the open-loop
	strategy of player $i$. For every strategy pair $(\su^1,\su^2)$, let
	$\bm{x}(\su^1,\su^2)$ denote the corresponding state trajectory, with
	$x_k(\su^1,\su^2)\in\mathbb{R}^n$ denoting its state at $k\in\K$. We
	further assume that the decision variables of player $i\in\N$ at each
	$k\in\K_l$ satisfy the following inequality constraints:
	\begin{align}
		g_k^i(x_k,u_k^1,u_k^2)
		:=
		M_k^ix_k+N_k^{i1}u_k^1+N_k^{i2}u_k^2+r_k^i
		\geq 0,
		\label{eq:affine_constraints}
	\end{align}
	where $M_k^i\in \mathbb R^{c_i\times n}$,
	$N_k^{ij}\in \mathbb R^{c_i\times m_j}$ for $j\in\N$, and
	$r_k^i\in \mathbb R^{c_i}$. Each player $i\in\N$ seeks to minimize the
	following stage-additive cost functional
	\begin{align}
		J_i(\su^1,\su^2)
		&=\tfrac12 \|x_K\|^2_{Q_K^i}\notag \\
		&\quad +\tfrac12\sum_{k\in\K_l}\big(
		\|x_k\|^2_{Q_k^i}
		+\|u_k^1\|^2_{R_k^{i1}}
		+\|u_k^2\|^2_{R_k^{i2}}\big),
		\label{eq:standard_cost}
	\end{align}
	\label{eq:LQDG}%
\end{subequations}
where $Q_k^i\in\mathbb S^n$, $k\in\K$, and
$R_k^{ij}\in\mathbb S^{m_j}$ for $i,j\in\N$ and $k\in\K_l$.

For an unconstrained LQDG, the follower's best response is affine in the
leader's announcement under suitable regularity conditions
\cite{Xu:2015}, \cite[Chapter 7]{Basar:1999}. This representation need not
apply under the coupled inequality constraints
\eqref{eq:affine_constraints}. In this setting, follower feasibility depends
on the leader's announcement, and the induced follower response may affect
leader feasibility. We therefore introduce the following sets 
to formulate
the constrained open-loop Stackelberg problem.

For a fixed leader announcement $\su^1\in \mathbb R^{Km_1}$, the feasible
open-loop strategy set of the follower is given by
\begin{align}
	\Omega_2(\su^1)&:=\Big\{\su^2\in \mathbb R^{Km_2}~|~ \notag\\
	&\quad g_k^2\big(x_k(\su^1,\su^2),u_k^1,u_k^2\big)\geq 0,~
	k\in \K_l \Big\}.
\end{align}
Then, the follower-feasible leader-announcement set is
\begin{align}
	\mathsf U_1
	:=
	\Big\{
	\su^1\in \mathbb R^{Km_1}~|~
	\Omega_2(\su^1)\neq \varnothing
	\Big\}.
\end{align}
For $\su^1\in \mathsf U_1$, the follower's rational reaction set is given by
\begin{align}
	\mathsf T(\su^1)
	:=
	\argmin_{\su^2\in \Omega_2(\su^1)}
	J_2(\su^1,\su^2).
	\label{eq:reactionset}
\end{align}
The response map \(\su^1\mapsto\mathsf T(\su^1)\) may be set-valued.
We impose the following regularity assumption. 
\begin{assumption}
	\label{assum:regularity}
	\begin{enumerate}
		\item $\mathsf U_1 \neq \varnothing$.
		\item For each $\su^1\in \mathsf U_1$, $\Omega_2(\su^1)$ is bounded.
	\end{enumerate}
\end{assumption}
Item~(1) is a joint-feasibility condition that can be checked by a linear
program. After eliminating the state through \eqref{eq:state_dyn},
\(\Omega_2(\su^1)\) is a closed convex polyhedron. Hence, item~(2) implies
that \(\Omega_2(\su^1)\) is compact.
The set of admissible leader announcements is given by
\begin{align}
	\Omega_1
	&:=
	\Big\{
	\su^1\in\mathsf U_1\ \Big|\
	\mathsf{T}(\su^1)\ \text{is single-valued}, \notag\\
	&\qquad
	g_k^1\big(
	x_k(\su^1,\mathsf{T}(\su^1)),
	u_k^1,\mathsf{T}_k(\su^1)
	\big)\geq0,\
	\forall k\in\K_l
	\Big\}.
	\label{eq:leader_feasible_set}
\end{align}
 Here, \(\mathsf T_k(\su^1)\in\mathbb R^{m_2}\) denotes the follower's
stage-\(k\) action induced by \eqref{eq:reactionset}.
\begin{remark}
	If $\mathsf T(\su^1)$ is not a singleton, the leader's cost depends on the
	follower's optimal response, which the leader cannot enforce. A worst-case
	selection of the follower's response therefore yields a well-defined but
	conservative problem. The literature commonly restricts attention to
	announcements with single-valued responses; see
	\cite[Theorems 7.1 and 7.2]{Basar:1999}. We adopt this convention in
	\eqref{eq:leader_feasible_set}.
\end{remark}

The leader's optimization problem is
\begin{align}
	\min_{\su^1\in \Omega_1} J_1(\su^1,\mathsf T(\su^1)).
	\label{eq:leaders_problem}
\end{align}
The generalized open-loop Stackelberg equilibrium for the LQDG
\eqref{eq:LQDG} is defined as follows.

\begin{definition}[Generalized open-loop Stackelberg equilibrium]
	\label{def:gse}
	A strategy profile $(\su^{1\star},\su^{2\star})$ is a generalized open-loop
	Stackelberg equilibrium (GOLSE) for the LQDG \eqref{eq:LQDG} if
	\begin{enumerate}[label=(\roman*)]
		\setlength\itemsep{.25em}
		\item $\su^{1\star}\in\Omega_1$ and
		$\su^{2\star}=\mathsf{T}(\su^{1\star})$;
		\item $
		J_1\big(\su^{1\star},\mathsf{T}(\su^{1\star})\big)
		\leq
		J_1\big(\su^1,\mathsf{T}(\su^1)\big),
		\quad \forall \su^1\in\Omega_1$.
	\end{enumerate}
\end{definition}

\section{Generalized Open-Loop Stackelberg Equilibrium}
\label{sec:GOLSE}
In this section, we provide an exact reformulation of the GOLSE as an
optimal-control problem subject to a discrete-time linear complementarity
system.
 
\subsection{Follower's Problem}

For a fixed leader announcement $\su^1\in\mathsf U_1$, the follower solves
the constrained optimal control problem
\begin{align}
	&\min_{\su^2\in\mathbb R^{Km_2}}~J_2(\su^1,\su^2)
	\label{eq:follower_ocp}\\
	\text{s.t. }&
	\eqref{eq:state_dyn},~~
	g_k^2(x_k,u_k^1,u_k^2)\geq0,~ k\in\K_l. \notag
\end{align}
The associated Lagrangian is
\begin{align*}
	\mathsf L_2
	={}&J_2(\su^1,\su^2)
	+\sum_{k\in\K_l}p_{k+1}^\top
	\Big(A_kx_k+B_k^1u_k^1+B_k^2u_k^2-x_{k+1}\Big) \notag\\
	&-\sum_{k\in\K_l}\mu_k^\top
	\Big(M_k^2x_k+N_k^{21}u_k^1+N_k^{22}u_k^2+r_k^2\Big)
	+p_0^\top(\bar{x}_0-x_0),
\end{align*}
where $p_k\in\mathbb R^n$, $k\in\K$, are the costate variables associated with
\eqref{eq:state_dyn}, and $\mu_k\in\mathbb R_+^{c_2}$, $k\in\K_l$, are the
Lagrange multipliers associated with \eqref{eq:affine_constraints}. The KKT
conditions are
\begin{subequations}\label{eq:follower_KKT}
	\begin{align}
		x_{k+1}
		&=A_kx_k+B_k^1u_k^1+B_k^2u_k^{2\star},~x_0=\bar{x}_0,
		\label{eq:state}\\
		p_k
		&=Q_k^2x_k+A_k^\top p_{k+1}-{M_k^2}^\top\mu_k,~p_K=Q^2_Kx_K,
		\label{eq:costate}\end{align} \begin{align}
		0
		&=R_k^{22}u_k^{2\star}+{B_k^2}^\top p_{k+1}
		-{N_k^{22}}^\top\mu_k,
		\label{eq:stationarity}\\
		0
		&\leq\mu_k\perp
		M_k^2x_k+N_k^{21}u_k^1+N_k^{22}u_k^{2\star}+r_k^2
		\geq0.
		\label{eq:followercomplementarity}
	\end{align}%
\end{subequations}
for $k\in\K_l$. Thus, \eqref{eq:follower_KKT} constitutes a coupled
discrete-time linear complementarity system (LCS); see \cite{Heemels:2000}.
We write
$\{\bm x,\su^1,\su^{2\star},\bm p,\bm\mu\}=
	\{x_k,p_k,k\in\K,~u_k^1,u_k^{2\star},\mu_k,k\in\K_l\}$ for a solution of this system,
when one exists.

Next, we provide conditions under which the follower's reaction set
\eqref{eq:reactionset} is single-valued. To this end, we make the following
assumption.
\begin{assumption}
	\label{assum:convexity}
	The solution of the following matrix Riccati difference equation (RDE)
	\begin{subequations}
		\begin{align}
			P_k&=Q_k^2+A_k^\top P_{k+1}A_k
			-V_k^\top\Gamma_k^{-1}V_k,\quad P_K=Q_K^2,
			\label{eq:riccati_P}\\
			\Gamma_k&:=R_k^{22}+{B_k^2}^\top P_{k+1}B_k^2,\quad 
			V_k:={B_k^2}^\top P_{k+1}A_k
			\label{eq:riccati_Gamma},
		\end{align}
		\label{eq:RDE}%
	\end{subequations}
	for $k\in\K_l$ exists, and the matrices $\{\Gamma_k,~k\in\K_l\}$ are
	positive definite.
\end{assumption}

For $k\in\K_l$, define
\begin{align*}
	&L_k^x:=-\Gamma_k^{-1}V_k,~~
	L_k^u:=-\Gamma_k^{-1}{B_k^2}^\top P_{k+1}B_k^1, \notag\\
	&L_k^\zeta:=-\Gamma_k^{-1}{B_k^2}^\top,~~
	L_k^\mu:=\Gamma_k^{-1}{N_k^{22}}^\top,\notag\\
	&\overline A_k:=A_k+B_k^2L_k^x,~~
	\overline B_k:=B_k^1+B_k^2L_k^u,~~
	\overline C_k:=B_k^2L_k^\mu,\\
	&\overline D_k:=B_k^2L_k^\zeta,~~
	\overline F_k:=A_k^\top P_{k+1}B_k^1+V_k^\top L_k^u,~~
	\overline G_k:=V_k^\top L_k^\mu-{M_k^2}^\top,\\
	&E_k^i:=M_k^i+N_k^{i2}L_k^x,~~
	S_k^i:=N_k^{i1}+N_k^{i2}L_k^u,\\
	&W_k^i:=N_k^{i2}L_k^\zeta,~~
	U_k^i:=N_k^{i2}L_k^\mu,\quad i\in\{1,2\}.
\end{align*} 
	\begin{lemma}
		\label{lem:follower}
		Let Assumptions~\ref{assum:regularity} and \ref{assum:convexity} hold.
		Then $\su^2\mapsto J_2(\su^1,\su^2)$ is strictly convex, and the follower's
		reaction set \eqref{eq:reactionset} is single-valued; consequently the KKT
		conditions \eqref{eq:follower_KKT} are necessary and sufficient for
		optimality in \eqref{eq:follower_ocp}.
		Further, 	$\{\bm x, \su^1, \su^{2\star}, \bm p, \bm \mu\}$
		satisfies the LCS \eqref{eq:follower_KKT} if and only if,
		with $\zeta_k=p_k-P_kx_k$,
		\begin{align}
			\mathsf T_k(\su^1)=u_k^{2\star}
			=L_k^xx_k+L_k^uu_k^1+L_k^\zeta\zeta_{k+1}+L_k^\mu\mu_k,
			\label{eq:follower_response}
		\end{align}
		and  $\{\bm x,\su^1,\bm \zeta,\bm \mu \}
		=	\{x_k,\zeta_k,k\in \K,~u_k^1,\mu_k,k\in \K_l\}
		$ satisfies the following LCS defined for $k\in \K_l$
		\begin{subequations}\label{eq:LCS}
			\begin{align}
				&x_{k+1}=\overline A_kx_k+\overline B_ku_k^1
				+\overline C_k\mu_k+\overline D_k\zeta_{k+1},
				\quad x_0=\bar{x}_0,\label{eq:LCS_state}\\
				&\zeta_k=\overline A_k^\top\zeta_{k+1}
				+\overline F_ku_k^1+\overline G_k\mu_k,
				\quad \zeta_K=0,\label{eq:LCS_costate}\\
				&0\leq\mu_k\perp
				E_k^2x_k+S_k^2u_k^1+W_k^2\zeta_{k+1}
				+U_k^2\mu_k+r_k^2\geq0.
				\label{eq:LCS_complementarity}
			\end{align}%
		\end{subequations}
	\end{lemma}
	
	\begin{proof}	 
		\underline{\emph{Strict convexity}:} Fix $\su^1\in \mathsf U_1$, and let
		$\bm x(\su^1,\su^2)$ and $\tilde{\bm x}(\su^1,\bm 0)$ be state trajectories
		generated by the strategy pairs $(\su^1,\su^2)$ and $(\su^1,\bm 0)$
		respectively. Define $\delta_k:=x_k-\tilde{x}_k$ for $k\in \K$. Then, it
		follows from \eqref{eq:state_dyn} that
		$\delta_{k+1}=A_k\delta_k+B_k^2u_k^2$ and $\delta_0=0$. Using this, the
		follower's cost \eqref{eq:standard_cost} is rewritten as
		\begin{align}
			&J_2(\su^1,\su^2)=J_2(\su^1,\bm 0)
			+\sum_{k=0}^K \tilde{x}_k^\top Q_k^2\delta_k \notag\\
			&\qquad +\tfrac12 \delta_K^\top Q_K^2\delta_K
			+\tfrac12 \sum_{k=0}^{K-1}
			\left(
			\delta_k^\top Q_k^2\delta_k+
			{u_k^2}^\top R_k^{22}u_k^2
			\right).
			\label{eq:modcost}
		\end{align}
		As $\delta_0=0$ and $P_K=Q_K^2$, the telescoping sum satisfies
		$\sum_{k=0}^{K-1}
		(\delta_{k+1}^\top P_{k+1}\delta_{k+1}
		-\delta_k^\top P_k\delta_k)
		=
		\delta_K^\top P_K\delta_K-\delta_0^\top P_0\delta_0
		=
		\delta_K^\top Q_K^2\delta_K$.
		Therefore, the quadratic terms in \eqref{eq:modcost} equal
		$\tfrac12\sum_{k=0}^{K-1}
		\big(
		\delta_{k+1}^\top P_{k+1}\delta_{k+1}
		-\delta_k^\top P_k\delta_k
		+\delta_k^\top Q_k^2\delta_k+
		{u_k^2}^\top R_k^{22}u_k^2
		\big)$.
		Using \(\delta_{k+1}=A_k\delta_k+B_k^2u_k^2\) and completing squares gives 
		$\tfrac12\sum_{k=0}^{K-1}
		\big(
		\delta_k^\top
		[-P_k+A_k^\top P_{k+1}A_k+Q_k^2
		-V_k^\top\Gamma_k^{-1}V_k]\delta_k
		+
		(u_k^2+\Gamma_k^{-1}V_k\delta_k)^\top\Gamma_k
		(u_k^2+\Gamma_k^{-1}V_k\delta_k)
		\big)$. Then, using \eqref{eq:riccati_P} in \eqref{eq:modcost} yields
		\begin{align*}
			J_2(\su^1,\bm 0)
			+\sum_{k=0}^K \tilde{x}_k^\top Q_k^2\delta_k +\tfrac12\sum_{k=0}^{K-1}
			\|u_k^2+\Gamma_k^{-1}V_k\delta_k\|_{\Gamma_k}^2.
		\end{align*}
The first term is independent of \(\su^2\), the second is linear in
\(\su^2\), and the final term is quadratic. Since
\(\Gamma_k\succ0\) for  \(k\in\K_l\), the final term is nonnegative
and vanishes only if
\(u_k^2+\Gamma_k^{-1}V_k\delta_k=0\) for every \(k\in\K_l\). Since
\(\delta_0=0\), this gives \(u_0^2=0\) and \(\delta_1=0\). Induction gives \(\su^2=\bm 0\). Hence, the quadratic term is positive definite in
\(\su^2\), and \(J_2(\su^1,\cdot)\) is strictly convex.
		Assumption~\ref{assum:regularity} makes \(\Omega_2(\su^1)\) nonempty and
		compact, so the follower problem has a unique minimizer. Because its
		constraints are affine, \eqref{eq:follower_KKT} is necessary and sufficient
		for optimality.
 
		\noindent 	\underline{\emph{Reduction}:}
 $\Rightarrow$ Let \eqref{eq:follower_KKT} hold and put
$\zeta_k:=p_k-P_kx_k$, so $\zeta_K=0$ by $p_K=Q_K^2x_K$ and
$P_K=Q_K^2$. Substituting
$p_{k+1}=P_{k+1}x_{k+1}+\zeta_{k+1}$ and \eqref{eq:state} into
\eqref{eq:stationarity} and collecting terms gives
\[
\Gamma_ku_k^{2\star}
=
-V_kx_k-{B_k^2}^\top P_{k+1}B_k^1u_k^1
-{B_k^2}^\top\zeta_{k+1}+{N_k^{22}}^\top\mu_k,
\]
which by $\Gamma_k\succ0$ is \eqref{eq:follower_response}. Substituting
\eqref{eq:follower_response} into \eqref{eq:state} yields
\eqref{eq:LCS_state}, and into $g_k^2$ yields
\eqref{eq:LCS_complementarity}. Substitution into \eqref{eq:costate}, followed by
\eqref{eq:riccati_P}, gives \eqref{eq:LCS_costate}.

\noindent
$\Leftarrow$ Conversely, let
$\{\bm x,\su^1,\bm\zeta,\bm\mu\}$ satisfy \eqref{eq:LCS}, and define
$u_k^{2\star}$ by \eqref{eq:follower_response} and
$p_k:=P_kx_k+\zeta_k$ for $k\in\K$. Equation \eqref{eq:LCS_state},
together with \eqref{eq:follower_response}, gives \eqref{eq:state}.
Since \(\zeta_K=0\) and \(P_K=Q_K^2\), the terminal condition
\(p_K=Q_K^2x_K\) holds.
Substituting \eqref{eq:LCS_state} and \eqref{eq:follower_response} into
$p_k=P_kx_k+\zeta_k$, and using \eqref{eq:riccati_P} and
\eqref{eq:LCS_costate}, gives \eqref{eq:costate}. Further,
\eqref{eq:follower_response} is equivalent to
\eqref{eq:stationarity}, while \eqref{eq:LCS_complementarity} is
equivalent to \eqref{eq:followercomplementarity}. Hence,
$\{\bm x,\su^1,\su^{2\star},\bm p,\bm\mu\}$ satisfies
\eqref{eq:follower_KKT}.
\end{proof}
 
\begin{remark}
	\label{rem:stagewise-uniqueness}
	Since \(\Gamma_k\succ0\),
	\(U_k^2=N_k^{22}\Gamma_k^{-1}{N_k^{22}}^\top\succeq0\). Hence, for fixed
	\((x_k,u_k^1,\zeta_{k+1})\), \eqref{eq:LCS_complementarity} is a monotone
	linear complementarity problem (LCP) in \(\mu_k\)
	\cite[Chapter~3]{Cottle:2009}. If
	\(\rank(N_k^{22})=c_2\leq m_2\), then \(U_k^2\succ0\). Since every positive
	definite matrix is a \(P\)-matrix, the stagewise LCP has a unique solution
	\(\mu_k\) \cite{Cottle:2009}.
\end{remark}
\begin{remark}
	\label{rem:nonuniqueness}
	Fix \(\su^1\). Lemma~\ref{lem:follower} guarantees uniqueness of the
	follower's strategy \(\su^{2\star}\), but not necessarily of the  
	variables \((\bm\zeta,\bm\mu)\). The coupled system
	\eqref{eq:LCS_state}--\eqref{eq:LCS_costate} can admit multiple admissible
	pairs \((\bm x,\bm\zeta)\). Thus, distinct admissible pairs
	\((\bm\zeta,\bm\mu)\) may represent the same unique follower strategy
	\(\su^{2\star}\).
\end{remark}
\subsection{Leader's problem}
We next reformulate the leader's problem exactly as an optimal control
problem with linear complementarity constraints (OCP--LCS). Define
\begin{align*}
	&\bm x_l:=\col{x_k}_{k\in\K_l},~
	\bm\zeta_r:=\col{\zeta_k}_{k\in\K_r},\\
	&H_k^1:=\Bdg\big(Q_k^1,R_k^{11},R_k^{12}\big),~z_k:=\col{x_k,u_k^1,\zeta_{k+1},\mu_k},\\
	&L_k:=\begin{bmatrix}
		\I_n & \Z_{n\times m_1} & \Z_{n\times n} & \Z_{n\times c_2}\\
		\Z_{m_1\times n} & \I_{m_1} & \Z_{m_1\times n} & \Z_{m_1\times c_2}\\
		L_k^x & L_k^u & L_k^\zeta & L_k^\mu
	\end{bmatrix},~ k\in\K_l.
\end{align*}
Whenever \eqref{eq:LCS} holds,
$L_kz_k=\col{x_k,u_k^1,\mathsf T_k(\su^1)}$. We define the
recovery map of the follower's response \eqref{eq:follower_response} by
\begin{align}
	\mathsf R(\bm x_l,\su^1,\bm\zeta_r,\bm\mu)
	:=
	\col{L_k^x x_k+L_k^u u_k^1+L_k^\zeta\zeta_{k+1}
		+L_k^\mu\mu_k}_{k\in\K_l}.
	\label{eq:recovery}
\end{align}
 
Substituting \eqref{eq:follower_response} into the leader's objective gives
\begin{align}
	\overline J_1(\bm x,\su^1,\bm\zeta,\bm\mu)
	&:=J_1\big(\su^1,\mathsf T(\su^1)\big)\notag\\
	&=\tfrac12x_K^\top Q_K^1x_K+
	\tfrac12\sum_{k\in\K_l}z_k^\top L_k^\top H_k^1L_kz_k.
	\label{eq:leadersobjective}
\end{align}

\begin{theorem}[OCP-LCS reformulation]
	\label{thm:exact}
	Suppose Assumptions~\ref{assum:regularity}
	and~\ref{assum:convexity} hold. Consider the following problem:
	\begin{subequations}\label{eq:leaderprogram}
	\begin{align}
		&\mathrm{OCP\text{-}LCS}:\quad \min_{\bm x,\su^1,\bm \mu,\bm \zeta}\quad \overline J_1(\bm x,\su^1,\bm \zeta,\bm \mu)\\
		\text{s.t. }
		&x_{k+1}=\overline A_kx_k+\overline B_ku_k^1
		+\overline C_k\mu_k+\overline D_k\zeta_{k+1},\label{eq:leadA}\\
		&\zeta_k=\overline A_k^\top\zeta_{k+1}
		+\overline F_ku_k^1+\overline G_k\mu_k,\label{eq:leadB}\\
		&0\leq\mu_k\perp E_k^2x_k+S_k^2u_k^1+W_k^2\zeta_{k+1}
		+U_k^2\mu_k+r_k^2\geq0,\label{eq:leadC}\\
		&E_k^1x_k+S_k^1u_k^1+W_k^1\zeta_{k+1}+U_k^1\mu_k+r_k^1\geq 0,\label{eq:leadD}\\
		&\quad k\in\K_l,~ x_0=\bar{x}_0,~\zeta_K=0.\notag
	\end{align}
\end{subequations}
	Let $\mathsf S_{\mathrm{OCP\text{-}LCS}}$ denote the set of global
	minimizers of \eqref{eq:leaderprogram}. Then the set
	$\mathsf S_{\mathrm{GOLSE}}$ of GOLSEs for the LQDG
	\eqref{eq:LQDG} is
	\begin{align}
		\mathsf S_{\mathrm{GOLSE}}
		&=
		\Big\{
		\big(\su^1,\mathsf R(\bm x_l,\su^1,\bm\zeta_r,\bm\mu)\big)
		~\big|~\notag \\
		&\qquad \quad (\bm x,\su^1,\bm\zeta,\bm\mu)
		\in\mathsf S_{\mathrm{OCP\text{-}LCS}}
		\Big\}.
		\label{eq:soleq}
	\end{align} 
\end{theorem}
	\begin{proof}
		Let $\mathsf F_{\mathrm{OCP\text{-}LCS}}$ denote the feasible set of
		\eqref{eq:leaderprogram}. For each $\su^1\in\mathsf U_1$,
		Assumption~\ref{assum:regularity} ensures that $\Omega_2(\su^1)$ is a
		nonempty bounded polyhedron. Hence, the follower problem
		\eqref{eq:follower_ocp} attains its minimum. By
		Lemma~\ref{lem:follower}, its objective is strictly convex in $\su^2$,
		and therefore the follower response $\mathsf T(\su^1)$ is unique.
		Moreover, the KKT conditions \eqref{eq:follower_KKT} are necessary and
		sufficient for follower optimality.
		
		\noindent
		\underline{\emph{Feasible OCP--LCS point $\Rightarrow$ feasibility for
				\eqref{eq:leaders_problem}}:}
		Let $(\bm x,\su^1, \bm\zeta, \bm\mu)\in
		\mathsf  F_{\mathrm{OCP\text{-}LCS}}$, and define
		$\su^2:=\mathsf R(\bm x_l,\su^1,\bm\zeta_r,\bm\mu)$ and
		$p_k:=P_kx_k+\zeta_k$ for $k\in\K$. Since $\zeta_K=0$ and
		$P_K=Q_K^2$, we have $p_K=Q_K^2x_K$. Together with
		$x_0=\bar{x}_0$, this gives the boundary conditions in
		\eqref{eq:follower_KKT}. 		
		By the construction in Lemma~\ref{lem:follower},
		\eqref{eq:leadA}--\eqref{eq:leadC}, together with the recovery formula
		for $\su^2$, are equivalent to the state equation, stationarity
		conditions, costate recursion, and follower complementarity conditions
		in \eqref{eq:follower_KKT}. Hence,
		$\{\bm x,\su^1,\su^2,\bm p,\bm\mu\}$ satisfies
		\eqref{eq:follower_KKT}. In particular,
		$\su^2\in\Omega_2(\su^1)$, so $\su^1\in\mathsf U_1$, and KKT
		sufficiency gives $\su^2=\mathsf T(\su^1)$.
		Further, substituting the recovery formula for $\su^2$ into the leader
		constraints shows that \eqref{eq:leadD} is equivalent to
		$g_k^1(x_k,u_k^1,u_k^2)\geq0$ for every $k\in\K_l$. Thus,
		$\su^1\in\Omega_1$. Finally,
		$L_kz_k=\col{x_k,u_k^1,u_k^2}$, and hence
		\begin{align}
			\overline J_1(\bm x,\su^1,\bm\zeta,\bm\mu)
			=
			J_1\big(\su^1,\mathsf T(\su^1)\big).
			\label{eq:valuepreserved}
		\end{align}
		
		\noindent
		\underline{\emph{Feasibility for \eqref{eq:leaders_problem}
				$\Rightarrow$ a feasible OCP--LCS point}:}
		Conversely, let $\su^1\in\Omega_1$, set
		$\su^2:=\mathsf T(\su^1)$, and let
		$\bm x=\bm x(\su^1,\su^2)$ be the associated state trajectory. Since
		$\su^2$ solves the follower problem, there exist $\bm p$ and
		$\bm\mu$ such that
		$\{\bm x,\su^1,\su^2,\bm p,\bm\mu\}$ satisfies
		\eqref{eq:follower_KKT}. Define $\zeta_k:=p_k-P_kx_k$. The terminal
		condition in \eqref{eq:follower_KKT}, together with $P_K=Q_K^2$,
		implies $\zeta_K=0$.
		Applying the transformations in Lemma~\ref{lem:follower} to the KKT
		system yields \eqref{eq:leadA}--\eqref{eq:leadC}. Since
		$\su^1\in\Omega_1$, \eqref{eq:leadD} also holds. Therefore,
		$(\bm x,\su^1,\bm\zeta,\bm\mu)\in
		\mathsf  F_{\mathrm{OCP\text{-}LCS}}$. Moreover,
		$L_kz_k=\col{x_k,u_k^1,u_k^2}$ implies that
		\eqref{eq:valuepreserved} holds for this feasible OCP--LCS point.
		
		\noindent
		\underline{\emph{Equivalence}:} The above two implications show that a leader decision
		\(\su^1\in\Omega_1\) is feasible for \eqref{eq:leaders_problem} if and
		only if it corresponds to a feasible OCP--LCS point. In either case, the
		corresponding objective values agree. Consequently,
		\eqref{eq:leaderprogram} and \eqref{eq:leaders_problem} have the same
		optimal value. Moreover, an OCP--LCS feasible point is a global
		minimizer if and only if its leader strategy minimizes
		$J_1(\su^1,\mathsf T(\su^1))$ over $\Omega_1$, equivalently, if and
		only if $(\su^1,\mathsf T(\su^1))$ is a GOLSE in the sense of
		Definition~\ref{def:gse}. Applying $\mathsf R$ from \eqref{eq:recovery} to the set of global
		OCP--LCS minimizers yields \eqref{eq:soleq}.
	\end{proof}
 
\begin{remark}
	Theorem~\ref{thm:exact} provides an exact reformulation, but does not
	itself guarantee existence of a GOLSE or a global OCP--LCS minimizer.
	Assumption~\ref{assum:regularity} ensures existence of the follower
	response for every $\su^1\in\mathsf U_1$. Existence at the leader level
	additionally requires $\Omega_1\neq\varnothing$ and that
	\eqref{eq:leaders_problem} admit a global minimizer. Moreover, the
	auxiliary variables $(\bm\zeta,\bm\mu)$ associated with a fixed leader
	strategy need not be unique, although each admissible pair recovers
	the unique follower response $\mathsf T(\su^1)$ by
	Remark~\ref{rem:nonuniqueness}. Accordingly, the solution-set relation \eqref{eq:soleq}
	is expressed through the image of
	$\mathsf  S_{\mathrm{OCP\text{-}LCS}}$ under $\mathsf R$, rather than
	through a bijection between the two solution sets.
\end{remark}
\begin{remark}
	\label{rem:nonconvexity}
	The complementarity condition \eqref{eq:leadC} imposes an either-or condition:
	for each $k\in\K_l$ and $i=1,\ldots,c_2$, either $\mu_{k,i}=0$ or
	$\big[E_k^2x_k+S_k^2u_k^1+W_k^2\zeta_{k+1}
	+U_k^2\mu_k+r_k^2\big]_i=0$, with both terms nonnegative. Hence, it defines
	a union of polyhedral active-set cases and makes the OCP--LCS feasible set
	generally nonconvex.
\end{remark}

\subsection{LCQP reformulation}
\label{sec:LCQP}

In this subsection, we reformulate the OCP--LCS as a large-scale
linear complementarity quadratic program (LCQP). For a fixed
initial state $\bar x_0$, we eliminate the state and costate variables
from \eqref{eq:leadA}--\eqref{eq:leadB}, thereby obtaining a problem in
the variables $(\su^1,\bm\mu)$.

Since \eqref{eq:leadB} does not depend on the state variables, it can
be solved backward from $\zeta_K=0$. Using the notation introduced in
the Appendix, this gives
$\zeta_k=\sum_{j=k}^{K-1}\Psi_{k,j}
(\overline F_ju_j^1+\overline G_j\mu_j)$ for $k\in\K_r$. Similarly,
solving \eqref{eq:leadA} forward from $x_0=\bar x_0$ gives
$x_k=\Phi_{k,0}\bar x_0+\sum_{i=0}^{k-1}\Phi_{k,i+1}
(\overline B_iu_i^1+\overline C_i\mu_i+\overline D_i\zeta_{i+1})$
for $k\in\K$. Substituting the costate expression into the state
recursion and stacking the resulting variables yields
\begin{subequations}\label{eq:statecostatevector}
	\begin{align}
		\bm\zeta_r
		&=\bm\Psi\bm F\su^1+\bm\Psi\bm G\bm\mu,
		\label{eq:aggzeta}\\
		\bm x_l
		&=\bm\Phi_0\bar x_0+\bm\Theta_u\su^1+\bm\Theta_\mu\bm\mu,
		\label{eq:aggstate}\\
		x_K
		&=\Phi_{K,0}\bar x_0+\bm\vartheta_u\su^1
		+\bm\vartheta_\mu\bm\mu.
		\label{eq:aggterminal}
	\end{align}
\end{subequations}
For every fixed $(\su^1,\bm\mu,\bar x_0)$,
\eqref{eq:statecostatevector} uniquely characterizes the state and
costate variables satisfying \eqref{eq:leadA}--\eqref{eq:leadB}.

\begin{theorem}[LCQP reformulation]
	\label{th:OLNElcqp}
	Let Assumptions~\ref{assum:regularity} and~\ref{assum:convexity} hold.
	Consider the problem
	\begin{subequations}\label{eq:LCQP}
		\begin{align}
			&\mathrm{LCQP}(\bar x_0):\quad
			\min_{\su^1,\bm\mu}\quad
			\tfrac12\left\|
			\bm\Delta \begin{bmatrix}\su^1\\\bm\mu\end{bmatrix}
			+\bm\xi_0\bar x_0
			\right\|_{\bm{\Sigma}}^2,
			\label{eq:lcqpobj}\\
			\mathrm{s.t.}\quad
			&\bm N^1\su^1+\bm V^1\bm\mu+\bm q^1(\bar x_0)\geq0,
			\label{eq:lcqpcon}\\
			&0\leq\bm\mu\perp
			\bm N^2\su^1+\bm V^2\bm\mu+\bm q^2(\bar x_0)\geq0.
			\label{eq:lcqpLCP}
		\end{align}
	\end{subequations}
	Then $(\su^{1\star},\su^{2\star})$ is a GOLSE of the LQDG
	\eqref{eq:LQDG} if and only if there exists $\bm\mu^\star$ such that
	$(\su^{1\star},\bm\mu^\star)$ is a global minimizer of
	\eqref{eq:LCQP} and
	$\su^{2\star}=
	\mathsf R(\bm x_l,\su^{1\star},\bm\zeta_r,\bm\mu^\star)$, where
	$\bm x_l$ and $\bm\zeta_r$ are recovered from
	\eqref{eq:statecostatevector} using
	$(\su^{1\star},\bm\mu^\star)$.
\end{theorem}

\begin{proof}
	Substituting \eqref{eq:statecostatevector} into
	\eqref{eq:leadC} and \eqref{eq:leadD} yields
	\eqref{eq:lcqpLCP} and \eqref{eq:lcqpcon}, respectively, with
	$\bm N^i$, $\bm V^i$, and $\bm q^i(\bar x_0)$ as defined in the
	Appendix. Moreover, substituting
	\eqref{eq:statecostatevector} into \eqref{eq:leadersobjective} yields
	\eqref{eq:lcqpobj}, since $L_kz_k$ stacks to
	$\col{\bm x_l,\su^1,
		\bm\Lambda_u\su^1+\bm\Lambda_\mu\bm\mu+\bm\Lambda_0\bar x_0}$ and
	$\bm{\Sigma }$ collects $Q_K^1$ and
	$\{H_k^1\}_{k\in\K_l}$ in the corresponding block order.
	Therefore, for fixed $\bar x_0$, feasible points of
	\eqref{eq:leaderprogram} and \eqref{eq:LCQP} are in bijection and
	their objective values agree. The result follows from
	Theorem~\ref{thm:exact}.
\end{proof}
	\begin{remark}
		\label{rem:qpcc}
		Problem \eqref{eq:LCQP} is parameterized by
		$\bar x_0\in\mathbb R^n$. Let
		$\mathsf  S_{\mathrm{LCQP}}(\bar x_0)$ denote its set of global
		minimizers. By Theorem~\ref{th:OLNElcqp}, a GOLSE exists for precisely
		those initial conditions in
		$\{\bar x_0\in\mathbb R^n\mid
		\mathsf  S_{\mathrm{LCQP}}(\bar x_0)\neq\emptyset\}$. Multiple LCQP
		minimizers need not yield distinct GOLSEs, since distinct multiplier
		vectors $\bm\mu$ may correspond to the same leader strategy and the
		same unique follower response; see Remark~\ref{rem:nonuniqueness}.
		\end{remark} 

	\begin{remark}
		When \(\bm{\Sigma}\succeq0\), the LCQP objective is convex and
		\eqref{eq:lcqpcon} is affine. Eliminating the state and costate trajectories
		introduces no additional nonconvexity; thus, the complementarity constraint
		\eqref{eq:lcqpLCP}, as noted in Remark~\ref{rem:nonconvexity}, is the only
		nonconvex component of the LCQP.
	\end{remark}

\begin{remark}
		For $k\in\K_l$, define the complementarity slack
	$s_k:=E_k^2x_k+S_k^2u_k^1+W_k^2\zeta_{k+1}+U_k^2\mu_k+r_k^2$.
	Given valid finite
	bounds \(\overline{\mu}_{k,i}\) and \(\overline{s}_{k,i}\),
	\(0\leq\mu_{k,i}\perp s_{k,i}\geq0\) is equivalently represented by
	\begin{align*}
		0\leq\mu_{k,i}\leq\overline{\mu}_{k,i}b_{k,i},\quad
		0\leq s_{k,i}\leq\overline{s}_{k,i}(1-b_{k,i}),\quad
		b_{k,i}\in\{0,1\},
	\end{align*}
	for \(i=1,\ldots,c_2\). Replacing the complementarity relations in the
	OCP--LCS or LCQP by these constraints yields an exact mixed integer quadratic program (MIQP), since the remaining constraints are affine and the objective is quadratic. 
	Such an MIQP can be solved to
	global optimality by standard branch-and-bound solvers, such as \texttt{Gurobi} \cite{gurobi}, when the optimality gap is certified. 
	MIQP formulations for OCP--LCS have been considered in \cite{Ayd23}, and
	LCQP solution methods in \cite{Hal22}.	
\end{remark}

	\section{Numerical Illustration}
	\label{sec:numerical}
\begin{figure}[h]
	\centering
     \begin{tikzpicture}[scale=0.475,transform shape,
 	>={Stealth[length=2.2mm]},
		nd1/.style  ={circle,draw=pOne,thick,minimum size=5mm,font=\Large ,
			fill=pOne!10,text=pOne!70!black},
		nd2/.style  ={circle,draw=pTwo,thick,minimum size=5mm,font=\Large,
			fill=pTwo!10,text=pTwo!70!black},
		relA/.style ={rectangle,draw=shared, thick,minimum size=9mm,
			rounded corners=0.5pt,font=\Large,fill=gray!12},
		relB/.style ={rectangle,draw=shared, thick,densely dashed,minimum size=9mm,
			rounded corners=0.5pt,font=\Large,fill=gray!12},
		viaA/.style ={->,line width=.75pt},                       
		viaB/.style ={->,line width=.75pt,densely dashed},        
		lb1/.style  ={font=\Large,text=pOne,inner sep=1.2pt,fill=white,
			fill opacity=0.9,text opacity=1,rounded corners=0.5pt},
		lb2/.style  ={font=\Large,text=pTwo,inner sep=1.2pt,fill=white,
			fill opacity=0.9,text opacity=1,rounded corners=0.5pt},
		tagS/.style ={font=\Large,text=shared,inner sep=1.6pt,
			draw=shared!55,fill=shared!8,rounded corners=1pt},
		tag1/.style ={font=\Large,text=pOne,inner sep=1.6pt,
			draw=pOne!55,fill=pOne!6,rounded corners=1pt},
		tag2/.style ={font=\Large,text=pTwo,inner sep=1.6pt,
			draw=pTwo!55,fill=pTwo!6,rounded corners=1pt},
		]
		
		\node[nd1]  (S1) at (0,1.9)    {$S_1$};
		\node[font=\Large,text=pOne] at ($(S1)+(0,0.92)$) {player 1};
		\node[nd2]  (S2) at (0,-1.9)   {$S_2$};
		\node[font=\Large,text=pTwo] at ($(S2)+(0,-0.92)$) {player 2};
		\node[relA] (N1) at (4.6,1.05)  {$N_1$};
		\node[relB] (N2) at (4.6,-1.05) {$N_2$};
		\node[nd1]  (T1) at (9.2,1.9)  {$T_1$};
		\node[nd2]  (T2) at (9.2,-1.9) {$T_2$};

		\begin{scope}[band/.style={line width=4pt,line cap=round}]
			\draw[band,shared!50] ($(S2)!0.75!(N1)$) -- ($(S1)!0.75!(N1)$);
			\node[font=\Large,text=shared] at ($(S1)!0.75!(N1)+(0,0.34)$) {$c^1$};
			\draw[band,shared!50] ($(S1)!0.75!(N2)$) -- ($(S2)!0.75!(N2)$);
			\node[font=\Large,text=shared] at ($(S2)!0.75!(N2)+(0,-0.34)$) {$c^2$};
			\draw[band,pOne!40] ($(N2)!0.75!(T1)$) -- ($(N1)!0.75!(T1)$);
			\node[font=\Large,text=pOne] at ($(N1)!0.75!(T1)+(0,0.34)$) {$\tilde c^{\,1}$};
			\draw[band,pTwo!40] ($(N1)!0.75!(T2)$) -- ($(N2)!0.75!(T2)$);
			\node[font=\Large,text=pTwo] at ($(N2)!0.75!(T2)+(0,-0.34)$) {$\tilde c^{\,2}$};
		\end{scope}
		
		\node[font=\Large,text=shared] at ($(N1)+(0,0.95)$) {$x^1_k$};
		\node[font=\Large,text=shared] at ($(N2)+(0,-0.95)$) {$x^2_k$};
		
		\draw[viaA,pOne] (S1) -- node[lb1,pos=0.42] {$v^{11}_k$} (N1);
		\draw[viaA,pTwo] (S2) -- node[lb2,pos=0.28] {$v^{21}_k$} (N1);
		\draw[viaA,pOne] (N1) -- (T1);
		\draw[viaA,pTwo] (N1) -- (T2);
		\draw[viaB,pOne] (S1) -- node[lb1,pos=0.28] {$v^{12}_k$} (N2);
		\draw[viaB,pTwo] (S2) -- node[lb2,pos=0.42] {$v^{22}_k$} (N2);
		\draw[viaB,pOne] (N2) -- (T1);
		\draw[viaB,pTwo] (N2) -- (T2);
		 
	\end{tikzpicture} 
    \caption{Network flow game with two players and two relay nodes.} 
    \label{fig:relay}
 \end{figure}
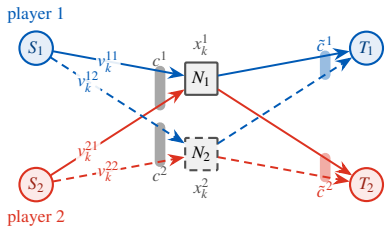 
 
We revisit the relay-network flow-control game in \cite{Moh23}. We compute
the GOLSE   and compare it with the generalized open-loop Nash equilibrium (GOLNE)
for the same network \cite{Moh23}. Thus, the two
equilibrium concepts differ only in the order of strategy selection. The
network has two sources \((S_1,S_2)\), two relays \((N_1,N_2)\), and two
destinations \((T_1,T_2)\); see Fig.~\ref{fig:relay}. At stage \(k\), player
\(i\in\{1,2\}\) routes flow \(v_k^{il}\) through relay \(N_l\), $l\in \{1,2\}$, whose state
evolves as \(x_{k+1}^l=x_k^l-\delta^l(v_k^{1l}+v_k^{2l})\), where
\(\delta^l>0\) is the depletion factor. Each player minimizes the negative
discounted payoff \(J^i\), given by
\begin{multline*}
	J^{i}=-\beta^K\sum_{l=1}^{2}
	\big(D^{il}x_K^{l}-\tfrac{1}{2}S^{il}(x_K^{l})^2\big)\\
	-\sum_{k=0}^{K-1}\sum_{l=1}^{2}\beta^k
	\Big(d^{il}x_k^{l}-\tfrac{1}{2}s^{il}(x_k^{l})^2
	+w^{il}v_k^{il}-\tfrac{1}{2}t^{il}(v_k^{il})^2\Big).
\end{multline*}
subject to shared relay-capacity and battery constraints
\(v_k^{1l}+v_k^{2l}\leq c^l\) and \(x_{k+1}^l\geq b_{\min}^l\), destination
constraints \(v_k^{i1}+v_k^{i2}\leq\tilde c^i\), and flow bounds
\(0\leq v_k^{il}\leq\bar v^{il}:=w^{il}/t^{il}\). We use
\(x_0^l=10\), \(b_{\min}^l=1\), \(D^{il}=S^{il}=8\), \(d^{il}=s^{il}=6\),
\(w^{il}=30\), \(t^{il}=10\), \(\delta^l=0.1\), \(c^l=5\), and
\(\tilde c^i=4\), for \(i,l\in\{1,2\}\), with \(K=30\) and \(\beta=0.95\).
The big-\(M\) MIQP reformulation of \eqref{eq:LCQP} was solved using
\texttt{Gurobi} with \(M=1000\) and \texttt{MIPGap}=\(10^{-8}\),
requiring \(1\)~s on an Intel Core i5-12500 (6-core) desktop with 32GB RAM.

 For \(e\in\{\mathrm N,\mathrm S\}\), where \(\mathrm N\) and \(\mathrm S\) denote the Nash and Stackelberg equilibria, respectively, let \(F_i^e:=\sum_{k=0}^{K-1}\sum_{l=1}^{2}v_{k,e}^{il}\) denote player \(i\)'s total delivered flow, \(F^e:=F_1^e+F_2^e\) the total delivered flow, and \(\rho_i^e:=F_i^e/F^e\) its flow share.  
 We further define player \(i\)'s allocation share at relay \(N_l\) as \(\alpha_{i,l}^e:=\sum_{k=0}^{K-1}v_{k,e}^{il}/\sum_{k=0}^{K-1}(v_{k,e}^{1l}+v_{k,e}^{2l})\), and the fraction of player \(i\)'s own traffic routed through \(N_l\) as \(\theta_{i,l}^e:=\sum_{k=0}^{K-1}v_{k,e}^{il}/F_i^e\), for \(i,l\in\{1,2\}\). 
 We quantify the effect of first-mover commitment through the leader's cost reduction \(\Delta_{\mathrm L}:=J^1_{\mathrm N}-J^1_{\mathrm S}\) and the follower's cost increase \(\Delta_{\mathrm F}:=J^2_{\mathrm S}-J^2_{\mathrm N}\).
 
We vary the initial charge of relay \(N_1\) over
\(x_0^1\in\{2,\ldots,18\}\), fixing \(x_0^2=10\) and all other parameters
at their baseline values. Thus, the initial relay charges are the only
source of asymmetry. The battery-supported flow ceiling,
\(S(x_0^1)=\sum_{l=1}^2\frac{x_0^l-b_{\min}^l}{\delta^l}
=10x_0^1+80\), reaches the aggregate destination-flow limit
\(K(\tilde{c}^1+\tilde{c}^2)=240\) at \(x_0^1=16\).

For low \(x_0^1\), \(N_1\) is scarce, and announcement enables the leader to
preempt its use. The flow profiles in Fig.~\ref{fig:Flow310} show that the
leader commits traffic to \(N_1\), inducing the follower to route more
traffic through \(N_2\). Accordingly,
\(\theta_{1,1}^{\mathrm S}>\theta_{1,1}^{\mathrm N}\) and
\(\theta_{2,1}^{\mathrm S}<\theta_{2,1}^{\mathrm N}\); see
Fig.~\ref{fig:Flowshare3}. The leader consequently obtains both a larger
aggregate flow share, \(\rho_1^{\mathrm S}>\rho_2^{\mathrm S}\), and a
larger share of \(N_1\) traffic,
\(\alpha_{1,1}^{\mathrm S}>\alpha_{2,1}^{\mathrm S}\); see
Figs.~\ref{fig:Flowshare1} and~\ref{fig:Flowshare2}. At \(x_0^1=3\),
player~1 receives approximately \(74.5\%\) of total flow and carries all
traffic through \(N_1\). Under Nash, in contrast, the own-route shares
remain symmetric, \(\theta_{1,l}^{\mathrm N}=\theta_{2,l}^{\mathrm N}\),
for \(l\in\{1,2\}\), although they vary with \(x_0^1\); see
Fig.~\ref{fig:Flowshare3}.

At the equal-charge point \(x_0^1=10\), each player splits its own
Stackelberg traffic evenly,
\(\theta_{i,1}^{\mathrm S}=\theta_{i,2}^{\mathrm S}=0.5\); see
Figs.~\ref{fig:Flow1010} and~\ref{fig:Flowshare3}. Hence,
\(\alpha_{i,1}^{\mathrm S}=\alpha_{i,2}^{\mathrm S}=\rho_i^{\mathrm S}\).
Nevertheless, the Stackelberg aggregate flow shares remain unequal:
\(\rho_1^{\mathrm S}\approx0.64\), whereas the Nash outcome has
\(\rho_1^{\mathrm N}=\rho_2^{\mathrm N}=0.5\); see
Figs.~\ref{fig:Flowshare1} and~\ref{fig:Flowshare2}. Thus, equal relay
charges eliminate within-player route preference, but not the aggregate
advantage of moving first.

As \(x_0^1\) increases further, the \(N_1\) allocation shares cross near
\(x_0^1=13\): player~2 becomes the primary user of \(N_1\), whereas
player~1 carries more \(N_2\) traffic; see Fig.~\ref{fig:Flowshare2}.
Stackelberg routing remains asymmetric, unlike the symmetric Nash routing:
the leader remains close to an equal relay split, while the follower
increasingly routes through \(N_1\); see
Fig.~\ref{fig:Flowshare3}. Thus, the crossover changes the relay-specific
form of the leader's advantage, but not its aggregate advantage.

Once the flow ceiling is reached at \(x_0^1=16\), the Stackelberg shares
become equal for \(x_0^1\geq16\):
\(\rho_1^{\mathrm S}=\rho_2^{\mathrm S}=0.5\); see
Fig.~\ref{fig:Flowshare1}. Routing nevertheless remains asymmetric:
\(\theta_{1,1}^{\mathrm S}=0.5\),
\(\theta_{2,1}^{\mathrm S}=0.75\), implying
\(\alpha_{1,1}^{\mathrm S}=0.4\) and
\(\alpha_{2,1}^{\mathrm S}=0.6\); see
Figs.~\ref{fig:Flowshare2} and~\ref{fig:Flowshare3}. The leader's cost
decreases and the follower's cost increases throughout the considered
range, approaching \(\Delta_{\mathrm L}\approx39.3\) and
\(\Delta_{\mathrm F}\approx117.8\), respectively; see
Fig.~\ref{fig:Costdiff}. Thus, even when aggregate flow shares become
equal, commitment continues to affect relay allocation, temporal routing,
and costs.
\begin{figure}[t]
	\centering
    \captionsetup[subfloat]{captionskip=-2pt}
	\subfloat[\scriptsize{Flow profiles at \((x_0^1,x_0^2)=(3,10)\).}]{ 
\begin{tikzpicture}[scale=0.565,transform shape]
	\begin{axis}[
		width=0.975\columnwidth,height=.6\columnwidth,
		grid=both,grid style={gray!20},
		xlabel style={font=\normalsize, yshift=5pt},
		xlabel=$k$,
		tick label style={font=\normalsize},label style={font=\normalsize},
		legend style={at={(axis cs:12,1)},font=\Large,draw=none,fill=none}, 
		ymin=0,ymax=3.25,xmin=0,xmax=29,]
    	\addplot[black,thick, line width = 1.5pt, draw opacity=0.9]table[x=stage,y=nash_v11,col sep=comma]
    	{relay_test_analysis/T4_flows_v3.csv};
    	
        \addplot[black,thick, dashed,  line width = 1.5pt, draw opacity=0.9]table[x=stage,y=nash_v12,col sep=comma]
    	{relay_test_analysis/T4_flows_v3.csv};
    	
    	\addplot[pOne,  line width = 1.5pt, draw opacity=0.9]table[x=stage,y=stack_v11,col sep=comma] {relay_test_analysis/T4_flows_v3.csv};
    	
    	\addplot[pOne,dashed, line width = 1.5pt,draw opacity=0.9]table[x=stage,y=stack_v12,col sep=comma]
    	{relay_test_analysis/T4_flows_v3.csv};
    	
    	\addplot[pTwo, line width = 2pt,draw opacity=0.9]table[x=stage,y=stack_v21,col sep=comma]
    	{relay_test_analysis/T4_flows_v3.csv};
    	
    	\addplot[pTwo,dashed, line width = 1.5pt,draw opacity=0.9]table[x=stage,y=stack_v22,col sep=comma]
    	{relay_test_analysis/T4_flows_v3.csv};
        \node[font=\normalsize, anchor=west, text=pOne] at (axis cs:3,1.75) {$v^{11}_{k,\mathrm{S}}$};
        \node[font=\normalsize, anchor=west, text=pOne] at (axis cs:12,3) {$v^{12}_{k,\mathrm{S}}$};
        \node[font=\normalsize, anchor=west, text=pTwo] at (axis cs:1,0.25) {$v^{21}_{k,\mathrm{S}}$};
        \node[font=\normalsize, anchor=west, text=pTwo] at (axis cs:10,1.5) {$v^{22}_{k,\mathrm{S}}$};
        \node[font=\normalsize, anchor=west, text=black] at (axis cs:16,1.5) {$v^{i2}_{k,\mathrm{N}}$};
        \node[font=\normalsize, anchor=west, text=black] at (axis cs:6.2,0.8) {$v^{i1}_{k,\mathrm{N}}$};
	\end{axis}
\end{tikzpicture} \label{fig:Flow310}}
	\subfloat[\scriptsize{Flow profiles at \((x_0^1,x_0^2)=(10,10)\).}]{\begin{tikzpicture}[scale=0.565,transform shape]
	\begin{axis}[
		width=0.975\columnwidth,height=.6\columnwidth,
		grid=both,grid style={gray!20},
		xlabel style={font=\normalsize, yshift=5pt},
		xlabel=$k$,
		tick label style={font=\normalsize},label style={font=\normalsize},
		legend style={at={(axis cs:12,1)},font=\Large,draw=none,fill=none}, 
		ymin=0,ymax=3.25,xmin=0,xmax=29,]
		\addplot[black, line width = 1.5pt, draw opacity=0.9] table[x=stage,y=nash_v11,col sep=comma]
		{relay_test_analysis/T4_flows_v10.csv};
		
		\addplot[black, dashed,  line width = 1.5pt, draw opacity=0.9]table[x=stage,y=nash_v12,col sep=comma]
		{relay_test_analysis/T4_flows_v10.csv};
		
		\addplot[pOne,  line width = 1.5pt, draw opacity=0.9]table[x=stage,y=stack_v11,col sep=comma]
		{relay_test_analysis/T4_flows_v10.csv};
		
		\addplot[pOne,dashed,  line width = 1.5pt, draw opacity=0.9]table[x=stage,y=stack_v12,col sep=comma]
		{relay_test_analysis/T4_flows_v10.csv};
		
		\addplot[pTwo, line width = 1.5pt, draw opacity=0.9]table[x=stage,y=stack_v21,col sep=comma]
		{relay_test_analysis/T4_flows_v10.csv};
		
		\addplot[pTwo,dashed, line width = 1.5pt, draw opacity=0.9]table[x=stage,y=stack_v22,col sep=comma]
		{relay_test_analysis/T4_flows_v10.csv};
		\node[font=\normalsize, anchor=west, text=pOne] at (axis cs:15,2.25) {$v^{11}_{k,\mathrm{S}}=v^{12}_{k,\mathrm{S}}$};
		\node[font=\normalsize, anchor=west, text=black] at (axis cs:20,1.4) {$v^{il}_{k,\mathrm{N}}$};
		\node[font=\normalsize, anchor=west, text=pTwo] at (axis cs:8,1) {$v^{21}_{k,\mathrm{S}}=v^{22}_{k,\mathrm{S}}$}; 
	\end{axis}
\end{tikzpicture} \label{fig:Flow1010}}\\
	\subfloat[\scriptsize{Total flow shares.}]{\begin{tikzpicture}[scale=0.5425,transform shape]
	\begin{axis}[width=.975\columnwidth,height=.6\columnwidth,
		xmin=2,	xmax=18, ymin=0,ymax=1.05,
		grid=both, 	grid style={gray!20},
		xlabel={$x_0^1$},
		xlabel style={font=\normalsize,xshift=-12pt,yshift=10pt},
		xlabel style={font=\normalsize,yshift=3pt},
		tick label style={font=\normalsize}, 
		every axis plot/.append style={line width=0.5pt,mark size=1.25pt}]
		
 		\addplot+[black, line width=1pt, dashed,mark=none]coordinates {(10,0) (10,1.05)};
 		\addplot+[black, line width=1pt, dashed,mark=none]coordinates {(16,0) (16,1.05)};
 
 		\addplot+[black,densely dashed,mark=none]table[x=x0_relay1,y=share1_nash,col sep=comma]
		{relay_test_analysis/T4_summary.csv};
		
		\addplot+[pOne,solid,mark=*]table[x=x0_relay1,y=share1_stack,col sep=comma]
		{relay_test_analysis/T4_summary.csv};
		
		\addplot+[pTwo,solid,mark=square*,mark options={fill=pTwo,draw=pTwo}]
		table[x=x0_relay1,y=share2_stack,col sep=comma]
		{relay_test_analysis/T4_summary.csv}; 
    	\node[font=\small, anchor=west, text=pOne] at (axis cs:3,0.85) {$\rho_1^{\mathrm{S}}$}; 
  		\node[font=\small, anchor=west, text=pTwo] at (axis cs:3,0.15) {$\rho_2^{\mathrm{S}}$}; 
    	\node[font=\small, anchor=west, text=black] at (axis cs:2.25,0.55) {GOLNE};
    	\node[font=\small, anchor=west, text=black] at (axis cs:10,0.75) {baseline}; 
    	\node[font=\small, anchor=east, text=black] at (axis cs:16,0.9) {flow ceiling}; 
   	
	\end{axis}
\end{tikzpicture}\label{fig:Flowshare1}}
	\subfloat[\scriptsize{Relay traffic shares.}]{\begin{tikzpicture}[scale=0.5425,transform shape]
	\begin{axis}[width=.975\columnwidth,height=.6\columnwidth,
		xmin=2,	xmax=18, ymin=-0.1,ymax=1.1,
		grid=both, 	grid style={gray!20},
		xlabel={$x_0^1$},
		xlabel style={font=\normalsize,xshift=-12pt,yshift=10pt},
		xlabel style={font=\normalsize,yshift=3pt},
		tick label style={font=\normalsize}, 
		every axis plot/.append style={line width=0.5pt,mark size=1.25pt}]
		
    	\addplot+[black,dashed ,line width =1pt, mark=none]coordinates {(10,-0.1) (10,1.1)};
 		\addplot+[black,dashed,line width=1pt, mark=none]coordinates {(13,-0.1) (13,1.1)};
 		\addplot+[black, line width=1pt, dashed,mark=none]coordinates {(16,-0.1) (16,1.1)};
 
 		\addplot+[black,densely dashed,mark=none]table[x=x0_relay1,y=capture1_N1_nash,col sep=comma]
		{relay_test_analysis/T4_summary.csv}; 
 
 		\addplot+[pOne,mark=*,mark options={fill=pOne,draw=pOne}]
		table[x=x0_relay1,y=capture1_N1_stack,col sep=comma]
		{relay_test_analysis/T4_summary.csv};
		\addplot+[pTwo, mark=square*,	mark options={fill=pTwo,draw=pTwo}]
		table[x=x0_relay1,y=capture2_N1_stack,col sep=comma]
		{relay_test_analysis/T4_summary.csv};

		\addplot+[pOne,dashed,mark=*,mark options={fill=pOne,draw=pOne}]
		table[x=x0_relay1,y=capture1_N2_stack,col sep=comma]
		{relay_test_analysis/T4_summary.csv};

 		\addplot+[pTwo,dashed,mark=square*,	mark options={fill=pTwo,draw=pTwo}]
 		table[x=x0_relay1,y=capture2_N2_stack,col sep=comma]
 		{relay_test_analysis/T4_summary.csv};

    	\node[font=\normalsize, anchor=west, text=pOne] at (axis cs:5,0.975) {$\alpha_{1,1}^{\mathrm{S}}$}; 
    	\node[font=\normalsize, anchor=west, text=pTwo] at (axis cs:5,0.025) {$\alpha_{2,1}^{\mathrm{S}}$}; 
    	
    	\node[font=\normalsize, anchor=west, text=pOne] at (axis cs:2,0.8) {$\alpha_{1,2}^{\mathrm{S}}$}; 
        \node[font=\normalsize, anchor=west, text=pTwo] at (axis cs:2,0.2) {$\alpha_{2,2}^{\mathrm{S}}$}; 
    	\node[font=\normalsize, anchor=west, text=black] at (axis cs:2.25,0.55) {GOLNE}; 
	\end{axis}
\end{tikzpicture}\label{fig:Flowshare2}}\\
	\subfloat[\scriptsize{Own traffic share.}]{\begin{tikzpicture}[scale=0.5425,transform shape]
	\begin{axis}[width=.975\columnwidth,height=.6\columnwidth,
		xmin=2,	xmax=18, ymin=0,ymax=1.05,
		grid=both, 	grid style={gray!20},
		xlabel={$x_0^1$},
		xlabel style={font=\normalsize,xshift=-12pt,yshift=10pt},
		xlabel style={font=\normalsize,yshift=3pt},
		tick label style={font=\normalsize}, 
		every axis plot/.append style={line width=.5pt,mark size=1.25pt}]
		
 		\addplot+[black,dashed,line width=1pt,mark=none]coordinates {(10,0) (10,1.05)};
  		\addplot+[black, line width=1pt, dashed,mark=none]coordinates {(16,0) (16,1.05)};
 		\addplot+[black,mark=diamond*,	mark options={fill=black,draw=black}]table[x=x0_relay1,y=route1_N1_nash,col sep=comma]
		{relay_test_analysis/T4_summary.csv};
 		\addplot+[black,mark=diamond*, dashed,	mark options={fill=black,draw=black}]table[x=x0_relay1,y=route1_N2_nash,col sep=comma]
 		{relay_test_analysis/T4_summary.csv}; 
 		\addplot+[pOne,mark=*,mark options={fill=pOne,draw=pOne}]
		table[x=x0_relay1,y=route1_N1_stack,col sep=comma]
		{relay_test_analysis/T4_summary.csv};
		\addplot+[pTwo, mark=square*,	mark options={fill=pTwo,draw=pTwo}]
		table[x=x0_relay1,y=route2_N1_stack,col sep=comma]
		{relay_test_analysis/T4_summary.csv}; 
		\addplot+[pOne,mark=*,dashed, mark options={fill=pOne,draw=pOne}]
		table[x=x0_relay1,y=route1_N2_stack,col sep=comma]
		{relay_test_analysis/T4_summary.csv};
		\addplot+[pTwo, mark=square*,dashed,mark options={fill=pTwo,draw=pTwo}]
		table[x=x0_relay1,y=route2_N2_stack,col sep=comma]
		{relay_test_analysis/T4_summary.csv};

  		\node[font=\normalsize, anchor=west, text=Black] at (axis cs:3.1,0.15) {$\theta_{i,1}^{\mathrm{N}}$}; 
  		\node[font=\normalsize, anchor=west, text=Black] at (axis cs:3.1,0.85) {$\theta_{i,2}^{\mathrm{N}}$};   		
    	\node[font=\normalsize, anchor=west, text=pOne] at (axis cs:2,0.35) {$\theta_{1,1}^{\mathrm{S}}$}; 
    	\node[font=\normalsize, anchor=west, text=pOne] at (axis cs:2,0.625) {$\theta_{1,2}^{\mathrm{S}}$}; 
  		\node[font=\normalsize, anchor=west, text=pTwo] at (axis cs:6,0.1) {$\theta_{2,1}^{\mathrm{S}}$}; 
 		\node[font=\normalsize, anchor=west, text=pTwo] at (axis cs:6,0.9) {$\theta_{2,2}^{\mathrm{S}}$};  
   	
	\end{axis}
\end{tikzpicture}\label{fig:Flowshare3}}	
	\subfloat[\scriptsize{Cost differences.}]{\begin{tikzpicture}[scale=0.5425,transform shape]
	\begin{axis}[width=.975\columnwidth,height=.6\columnwidth,
	xmin=2,	xmax=18, ymin=0,ymax=450,
	grid=both, 	grid style={gray!20},
	xlabel={$x_0^1$},
	xlabel style={font=\normalsize,xshift=-12pt,yshift=10pt},
	xlabel style={font=\normalsize,yshift=3pt},
	tick label style={font=\normalsize}, 
	every axis plot/.append style={line width=0.5pt,mark size=1.25pt}]
		
		\addplot+[black,dashed,line width=1pt,mark=none]coordinates {(10,0) (10,450)};
		\addplot+[black,dashed,line width=1pt,mark=none]coordinates {(16,0) (16,450)};
		\addplot+[pOne,mark=*,,mark options={fill=pOne,draw=pOne},draw opacity=1]
		table[x=x0_relay1,y=leader_advantage,col sep=comma]
		{relay_test_analysis/T4_summary.csv};
		
		\addplot+[pTwo,mark=square*,mark options={fill=pTwo,draw=pTwo},draw opacity=1]
		table[x=x0_relay1,y=follower_loss,col sep=comma]
		{relay_test_analysis/T4_summary.csv};
		
		\node[font=\normalsize, anchor=west, text=pOne] at (axis cs:4,200) {$\Delta_L$}; 
		\node[font=\normalsize, anchor=west, text=pTwo] at (axis cs:5,340) {$\Delta_F$}; 
		\node[font=\normalsize, anchor=west, text=black] at (axis cs:10,320) {baseline}; 
		\node[font=\normalsize, anchor=east, text=black] at (axis cs:16,400) {flow ceiling}; 
	\end{axis}
\end{tikzpicture}\label{fig:Costdiff}}
	\caption{Strategic routing under asymmetric initial relay charges.}	
	\label{fig:T4}
\end{figure}
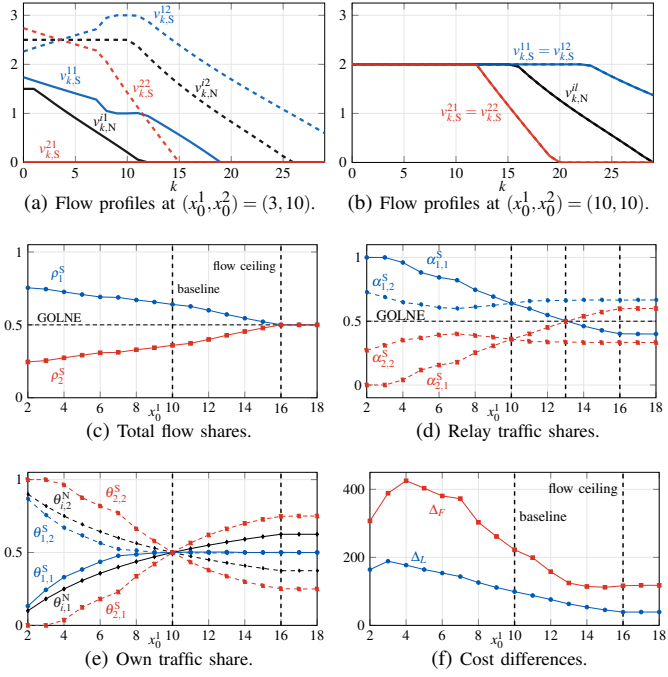 
 
\section{Conclusion}\label{sec:conclusion}
We studied finite-horizon linear--quadratic Stackelberg difference games
with coupled-affine state-control constraints. Using a Riccati-based
costate transformation, we derived an exact OCP--LCS reformulation and a
large-scale LCQP parametrized by the initial state. Global minimizers of this LCQP yield generalized open-loop Stackelberg
equilibria, with the follower strategy obtained from an explicit affine
recovery map. A natural direction for future work is to
extend the results to a feedback information structure.
	\appendix 
	\subsubsection*{Notations for Section \ref{sec:LCQP}}
We first define the state transition matrices associated with the forward and backward equations \eqref{eq:leadA}--\eqref{eq:leadB} and then define the block matrices as
	\begin{align*}
		&\Psi_{k,j}:=\overline A_k^\top\overline A_{k+1}^\top\cdots\overline A_{j-1}^\top,
		~k\leq j,~\Psi_{k,k}=\I_n,~k,j\in \K_l, \notag\\
		&\Phi_{k,i}:=\overline A_{k-1}\overline A_{k-2}\cdots\overline A_i,
		~k\geq i,~\Phi_{k,k}=\I_n,~i,k\in \K, \notag\\
&\bm \Phi_0=\col{\Phi_{k,0}}_{k\in \K_l},\,\bm B=\Bdg\{\overline B_k\}_{k\in \K_l},\, \bm C=\Bdg\{\overline C_k\}_{k\in \K_l},\\
&\bm D=\Bdg\{\overline D_k\}_{k\in \K_l},\,
\bm F=\Bdg\{\overline F_k\}_{k\in \K_l},\,
\bm G=\Bdg\{\overline G_k\}_{k\in \K_l}.\end{align*} 
	We define the following matrices, with block indices $i,j\in \K_r$,
	\begin{align*}
		&[\bm \Psi]_{ij}:=\begin{cases} \Z_{n\times n},&~i\geq j\\ \Psi_{i,j-1},&~i<j\end{cases},\qquad
		[\bm \Phi]_{ij}:=\begin{cases} \Phi_{i-1,j},&~i>j\\ \Z_{n\times n},&~i\leq j\end{cases},
		\\
		& \text{for $i\in \N$},~ \bm E^i= \Bdg\{E^i_k\}_{k\in \K_l},~
		\bm S^i=\Bdg\{S^i_k\}_{k\in \K_l},\\
		&\bm W^i=\Bdg\{W^i_k\}_{k\in \K_l},~
		\bm U^i=\Bdg\{U^i_k\}_{k\in \K_l},~\bm r^i=\col{r^i_k}_{k\in \K_l},\\
		&\bm T^i:=\bm E^i\bm \Phi \bm D+\bm W^i,~\bm q^i(\bar x_0):=\bm E^i\bm \Phi_0\bar x_0+\bm r^i,\\
		&\bm N^i:=\bm E^i\bm \Phi \bm B+\bm S^i+\bm T^i\bm \Psi \bm F,~
		\bm V^i:=\bm E^i\bm \Phi \bm C+\bm U^i+\bm T^i\bm \Psi \bm G,\\
		&\bm \Theta_u:=\bm \Phi \bm B+\bm \Phi \bm D\bm \Psi \bm F,~\bm \Theta_\mu:=\bm \Phi \bm C+\bm \Phi \bm D\bm \Psi \bm G,\\
		&\bm \Phi_K:=\row{\Phi_{K,k}}_{k\in \K_r},~\bm L^\alpha=\Bdg\{L^\alpha_k\}_{k\in \K_l},  ~\alpha\in\{x,u,\zeta,\mu\},\\
		&\bm \Lambda_u:=\bm L^x\bm \Theta_u+\bm L^u+\bm L^\zeta\bm \Psi \bm F,~
		\bm \Lambda_\mu:=\bm L^x\bm \Theta_\mu+\bm L^\zeta\bm \Psi \bm G+\bm L^\mu,\\
		& 
		\bm \Lambda_0:=\bm L^x\bm \Phi_0, 
		\bm \vartheta_u:=\bm \Phi_K(\bm B+\bm D\bm \Psi \bm F), 
		\bm \vartheta_\mu:=\bm \Phi_K(\bm C+\bm D\bm \Psi \bm G),\\
		& \bm Q^1=\Bdg\{Q_k^1\}_{k\in \K_l},
		\bm R^{11}=\Bdg\{R_k^{11}\}_{k\in \K_l},
		\bm R^{12}=\Bdg\{R_k^{12}\}_{k\in \K_l},\\
		&\bm{\Sigma}=\Bdg\bigl(Q^1_K,\bm Q^1,\bm R^{11},\bm R^{12}\bigr),\bm \xi_0=\col{\Phi_{K,0}, \bm \Phi_0,\Z_{Km_1\times n},\bm \Lambda_0}\\ 
	&\bm \Delta=\row{\col{\bm \vartheta_u,\bm \Theta_u,\I_{Km_1},\bm \Lambda_u},\col{\bm \vartheta_\mu,\bm \Theta_\mu,\Z_{Km_1\times Kc_2},\bm \Lambda_\mu}}\end{align*}
 
	\bibliographystyle{ieeetr}
	\bibliography{GSE_References}
\end{document}